\documentclass[letterpaper,10pt,conference,final]{ieeeconf}
\IEEEoverridecommandlockouts
\makeatletter
\let\NAT@parse\undefined
\makeatother

\usepackage[numbers,sort&compress]{natbib}
\usepackage{psfrag,amsmath,amssymb,amsbsy,graphicx,color,algorithm,algorithmic,xspace,pstricks,ifthen}
\usepackage[us]{datetime}
\usepackage[inline]{asymptote}

\newboolean{printheader}
\setboolean{printheader}{false}
\newboolean{printsubmitted}
\setboolean{printsubmitted}{false}
\ifthenelse{\boolean{printheader}}{%
\usepackage{fancyhdr}
\pagestyle{plain}
\fancyhead{}
\fancyfoot{}

\fancyhead[C]{{\large\tt\today\qquad\currenttime (European DST)}}
\fancyfoot[C]{\thepage}
}{}
\ifthenelse{\boolean{printsubmitted}}{%
\usepackage{fancyhdr}
\fancyfoot{}

\fancyfoot[CF]{{\large\tt Submitted to Conference on Decision and Control (CDC), 2012}}
}{}

\newcommand{\R}{\mathbb{R}}

\newcommand{\Mtt}{\ensuremath{M_{\text{TT}}}\xspace}
\newcommand{\Met}{\ensuremath{M_{\text{ET}}}\xspace}
\newcommand{\Mbus}{\ensuremath{\mathcal{M}_{\text{Bus}}(k)}\xspace}

\newcommand{\Tdw}{\ensuremath{T_{dw}}\xspace}
\newcommand{\Tdws}{\ensuremath{T_{dw}^*}\xspace}

\renewcommand{\eth}{\ensuremath{e_{\text{th}}}\xspace}
\newcommand{\qq}{\ensuremath{q^{-1}}}
\newcommand{\qn}[1]{\ensuremath{q^{-#1}}}
\newcommand{\yref}{\ensuremath{y_{\text{ref}}}\xspace}
\newcommand{\yrefp}{\ensuremath{y_{\text{ref}}'}\xspace}
\newcommand{\NN}{\ensuremath{\mathbb{N}}}

\newcommand{\defi}{\ensuremath{\mathrel{\mathop:}=}\xspace}

\newlength{\coorthwidth}
\newlength{\widthsim}
\newdimen\simdimen
\simdimen=\widthsim
\newdimen\widthdimen
\newdimen\abstand

\def\coorth#1{
\settowidth{\coorthwidth}{\ensuremath{#1}}
\widthdimen=\coorthwidth
\advance\widthdimen by 5pt
\divide\widthdimen by \simdimen
\abstand=\widthdimen
\multiply\abstand by 2
\mkern-\number\abstand mu\ensuremath{\stackrel{#1}{\scalebox{\number\widthdimen}[1]{$\sim$}}}\mkern-\number\abstand mu
}

\DeclareMathOperator*{\rank}{rank}
\DeclareMathOperator*{\dimm}{dim}
\DeclareMathOperator*{\imm}{Im}

\newtheorem{thm}{Theorem}
\newtheorem{lemma}{Lemma}
\newtheorem{cor}{Corollary}
\newtheorem{ass}{Assumption}
\newtheorem{definition}{Definition}

\newcounter{stepstages}
\stepcounter{stepstages}
\newcounter{stepproof}	
\newenvironment{proofstep}[1]{\stepcounter{stepproof}\noindent\textbf{Step \arabic{stepstages}-\arabic{stepproof}:} \textit{#1}

\underline{Proof of Step \arabic{stepstages}-\arabic{stepproof}:}}{}

\begin{document}

\bstctlcite{removedashes:BSTcontrol}

\title{Adaptive Switching Controllers for Tracking using Persistent Excitation}

\author{Harald Voit, Anuradha Annaswamy}%

\maketitle

\ifthenelse{\boolean{printheader}\OR\boolean{printsubmitted}}{\thispagestyle{fancy}}{}

\begin{abstract}
The focus of this paper is on the co-design of control and communication protocol for the control of multiple applications with unknown parameters using a distributed embedded system. The co-design consists of an adaptive switching controller and a hybrid communication architecture that switches between a time-triggered and event-triggered protocol. It is shown that the overall co-design leads to an overall switching adaptive system that has bounded solutions and ensures tracking in the presence of a class of disturbances.
\end{abstract}

\section{Introduction}
\label{sec:introduction}
Embedded control systems are ubiquitous and can be found in several applications including aircraft, automobiles, process control, and buildings. An embedded control system is one in which the computer system is designed to perform dedicated functions with real-time computational constraints \citep{Liu2000}. Typical features of such embedded control systems are the control of multiple applications, the use of shared networks used by different components of the systems to communicate with each other for control, a large number of sensors as well as actuators, and their distributed presence in the overall system. 

The most common feature of such distributed embedded control systems (DES) is  shared resources. Constrained by space, speed, and cost, often information has to be transmitted using a shared communication network. In order to manage the flow of information in the network, protocols that are time-triggered \cite{Nghiem2006} and event-triggered \cite{Cervin2006,Wang2011,Tabuada2007} have been suggested over the years. Associated with each of these communication protocols are different set of advantages and disadvantages. The assignment of time-triggered (TT) slots to all control-related signals has the advantage of high quality of control (QoC) due to the possibility of reduced or zero delays, but leads to poor utilization of the communication bandwidth, high cost, overall inflexibility, and infeasibility as the number of control applications increase. On the other hand, event-triggered (ET) schedules often result in poor control performance due to the unpredictable temporal behavior of control messages and the related large delays which occurs due to the lack of availability of the bus. These imply that a hybrid protocol that suitably switches between these two schedules offers the possibility of exploiting their combined advantages of high QoC, efficient resource utilization, and low cost \cite{Masrur2011}. 
Such a hybrid protocol is the focus of this paper. To combine the advantage of TT and ET policies, hybrid protocols are increasingly being studied in recent years. Examples of such protocols are FlexRay and TTCAN \cite{Albert2004,Talbot2009}, used extensively in automotive systems.

While several papers have considered control using TT protocols (see for example, \cite{Nghiem2006,Palopoli2005}) and ET protocols (see for example, \cite{Wang2011,Tabuada2007}), control using hybrid protocols has not been studied in the literature until recently. The co-design problem has begun to be addressed of late as well (see for example, \cite{Arzen2005b,Branicky2002,Cervin2003,Marti2004,Seto1996,Abeni2000,Naghshtabrizi2009,Samii2009}). In \cite{Seto1996,Abeni2000,Naghshtabrizi2009,Samii2009}, the design of scheduling policies that ensure a good Quality of Control (QoC) is addressed. In \cite{Seto1996}, the schedulability analysis of real-time tasks with respect to the stability of control functions is discussed. In \cite{Abeni2000}, modeling the real-time scheduling process as a dynamic system, an adaptive self-tuning regulator is proposed to adjust the bandwidth of each single task in order to achieve an efficient CPS utilization. The focus of most of the papers above are either on a simple platform or on a single processor. A good survey paper on co-design can be found in~\cite{Xia2006}. Our focus in this paper is on the co-design of adaptive switching controllers and hybrid protocols so as to ensure good tracking in the presence of parametric uncertainties in the plant being controlled while utilizing minimal resources in the DES.

The hybrid protocol that is addressed in this paper switches between a TT and a ET scheme. The TT scheme, which results in a negligible delay in the processing of the control messages, is employed when a control action is imperative and the ET scheme, which typically results in a non-zero delay, is employed when the controlled system is well-behaved, with minimal tracking error. The latter is in contrast to papers such as \cite{Wang2011,Tabuada2007} and \cite{Velasco2008} where the underlying \emph{event} is associated with a system error exceeding a certain threshold, while here an \emph{event} corresponds to the case when the system error is small. The controller is to be designed for multiple control applications, each of which is subjected to a parametric uncertainty. An adaptive switching methodology is introduced to accommodate these uncertainties and the hybrid nature of the protocol.

Switched control systems and related areas of hybrid systems and supervisory control have received increased attention in the last decade (see e.g., \cite{Narendra1997,Branicky1998,Liberzon1999,Hespanha1999,Narendra2000,Liberzon2003,McCourt2010,Rajhans2009}) and used in several applications (see e.g. \cite{Gao2007,Xie2010,McClamroch1996,Johansson2003}). Adaptive switched and tuned systems have been studied as well (see \cite{Narendra1997,Narendra2000,Morse1996}). The combined presence of uncertainties and switching delays makes a direct application of these existing results to the current problem inadequate. 

The solution to the problem of co-design of an adaptive swtiched controller and switches in a hybrid protocol was partially considered in \cite{Voit2012}, where the control goal was one of stabilization. In this paper, we consider tracking, which is a non-trivial extension of \cite{Voit2012}. The main reason for this lies in the trigger for the switch, which corresponds to a system error becoming small. In order to ensure that this error continues to remain small even in the presence of a non-zero reference signal, we needed to utilize fundamental properties of  the adaptive system with persistent excitation, and derive additional properties in the presence of reference signals with an invariant persistent excitation property. These properties in turn are suitably exploited and linked with the switching instants, and constitute the main contribution of this paper.

In Section \ref{sec:problem} the problem is formulated, and preliminaries related to adaptive control and persistent excitation are presented. In Section \ref{sec:switchingadaptivecontroller}, the switching adaptive controller is described and the main result of global boundedness is proved. Concluding remarks are presented in section \ref{sec:conclusion}.

\section{Problem Formulation}
\label{sec:problem}

\subsection{The plant model}\label{sec:plantmodel}

The problem that we address in this paper is the simultaneous control of $n$ plants, $C_i$, $i=1,\ldots,n$, in the presence of impulse disturbances that occur sporadically, using a hybrid communication protocol. We assume that each of these $n$ applications have the following problem statement.

The plant to be controlled is assumed to have a discrete time model described by
\begin{multline}
	\label{eq:model}
	\mathcal{C}_i:\;y(k)=-\sum_{l=1}^{m_1}a_ly(k-l)\\[-2.0ex]
	+b_0u(k-d)+\sum_{l=1}^{m_2}b_lu(k-l-d)+D(k-d)
\end{multline}
where $u(k)$ and $y(k)$ are the input and output of the $i$-th control application, respectively, at the time-instant $t_k$ and $d\geqslant 1$ is a time-delay. The disturbance $D(k)$ are assumed to be impulses that can occur occasionally with their inter-arrival time lower-bounded by a finite constant. The parameters of the $i$-th plant are given by $a_l$, $l=1,\ldots,m_1$, $b_l$,$l=0,\ldots,m_2$ and are assumed to be unknown. It is further assumed that the sampling time of the controller is a constant $h$, so that $t_{k+1}=t_k+h$. The goal is to choose the control input $u$ such that $y(k)$ tracks a desired signal $\yref(k)$, with all signals remaining bounded.

The model in (\ref{eq:model}) can be expressed as
\begin{equation}
\label{eq:leftdiff}
A(\qq)y(k)=\qn{d}(B(\qq)u(k)+D(k));\qquad k\geq 0
\end{equation}
where \qq\ is the backward shift operator and the polynomials $A$ and $B$ are given by
\begin{equation}
\label{eq:polyAB}
\begin{split}
A(\qq)=1+\sum_{l=1}^{m_1}a_l\qn{l}\qquad B(\qq)=b_0+\sum_{l=1}^{m_2}b_l\qn{l}
\end{split}
\end{equation}
The following assumptions are made regarding the plant poles and zeros:
\begin{ass}
1) An upper bound for the orders of the polynomials in (\ref{eq:polyAB}) is known and 2) all zeros of $B_i(\qq)$ lie strictly inside the closed unit~disk.
\label{ass:fixeddelay}
\end{ass}

For any delay $d$, Eq. (\ref{eq:model}) can be expressed in a \emph{predictor form} as follows \cite{Goodwin1984}:
\begin{equation}
\label{eq:predictorform}
y(k+d)=\alpha(\qq)y(k)+\beta(\qq)u(k)+D(k)
\end{equation}
with
\begin{equation}
\label{eq:alphabeta}
\begin{split}
\alpha(\qq)&=\alpha_0+\alpha_1\qq+\ldots+\alpha_{m_1-1}\qn{(m_1-1)}\\
\beta(\qq)&=F(\qq)B(\qq)\\
&=\beta_0+\beta_1\qq+\ldots+\beta_{m_2+d-1}\qn{(m_2+d-1)}
\end{split}
\end{equation}
where $F(\qq)$ and $\alpha(\qq)$ are the unique polynomials that satisfy the equation
\begin{equation}
\label{eq:diophantine}
1=F(\qq)A(\qq)+\qn{d}\alpha(\qq).
\end{equation}

Equation (\ref{eq:predictorform}) can be expressed as 
\begin{align}
  \label{eq:regressor}
  y(k+d)&=\theta_d^{*T}\Phi_d(k)+D(k)\\
  &=\vartheta_d^*\phi_d(k)+\beta_0^du(k)+D(k)\label{eq:thetaregressor}
\end{align}
where $\phi_d(k)$, $\vartheta_d^*$, $\Phi_d(k)$, and $\theta_d^*$ are defined as
\begin{equation}
  \label{eq:phiandthetaklein}
  \phi_d(k)=\begin{bmatrix}y(k)\\ \vdots\\y(k-m_1+1)\\u(k-1)\\\vdots\\u(k-m_2-d+1)\end{bmatrix}
  \quad
  \vartheta_d^*=\begin{bmatrix}\alpha_0^d\\\vdots \\\alpha_{m_1-1}^d\\\beta_1^d\\\vdots\\\beta_{m_2+d-1}^d\end{bmatrix}
\end{equation}
\begin{equation}
  \label{eq:phiandtheta}
  \Phi_d(k)=\begin{bmatrix}\phi_d(k)\\u(k)\end{bmatrix}
  \text{ and }
  \theta_d^*=\begin{bmatrix}\vartheta_d^*\\\beta_0^d\end{bmatrix}
\end{equation}
with $\phi_d(k)\in\mathbb{R}^{m_1+m_2+d-1}$, $\vartheta_d^*\in\mathbb{R}^{m_1+m_2+d-1}$, $\Phi_d(k)\in\mathbb{R}^{m_1+m_2+d}$, $\theta_d^*\in\mathbb{R}^{m_1+m_2+d}$, and $\alpha_j^d,j=0,\ldots,m_1-1$ and $\beta_j^d,j=0,\ldots,m_2+d-1$ the coefficients of the polynomials in (\ref{eq:alphabeta}) with respect to the delay $d$ and finite initial conditions
\begin{equation}
\begin{split}
y(k-i)&=y_0(i)\quad i=0,\ldots,m_1-1,\\
u(k-i)&=u_0(i)\quad i=1,\ldots,m_2+d-1.
\end{split}
\label{eq:initialconditions}
\end{equation}

From Eqs. (\ref{eq:regressor})-(\ref{eq:phiandtheta}), we observe that a feedback controller of the form
\begin{equation}
  \label{eq:feedbackcontroller}
  u(k)=\frac{1}{\beta_0^d}\left(\yref(k+d)-\vartheta_d^{*T}\phi_d(k)\right)
\end{equation}
realizes the objective of stability and follows the desired bounded trajectory $\yref(k)$ in the absence of disturbances. Designing a stabilizing controller $u(k)$ essentially boils down to a problem of implementing (\ref{eq:feedbackcontroller}) with the controller gain $\vartheta_d^*$. Two things should be noted: (i) Controller (\ref{eq:feedbackcontroller}) is not realizable as $\vartheta_d^*$ and $\beta_0^d$ are not known, and (ii) the dimension of $\phi_d(k)$, $\vartheta_d^*$ as well as the entries of $\vartheta_d^*$ depend on the delay~$d$. 

\subsection{Baseline adaptive controller}
\label{sec:adaptivecontrollerdesign}
Since $\vartheta_d^*$ and $\beta_0^d$ are unknown, we replace them with their parameter estimates and derive the following adaptive control input
\begin{equation}
\label{eq:adaptivecontroller}
u(k)=\frac{1}{\hat\theta_{d,\nu}(k)}\left(\yref(k+d)-\hat\vartheta_d(k)^{T}\phi_d(k)\right)
\end{equation}
where $\hat\theta_{d,\nu}(k)$ denotes the $(m_1+m_2+d)$-th element of the parameter estimation $\hat\theta_d(k)$ and is the estimate of $\beta_0^d$. $\hat\theta_d(k)$ is adjusted according to the adaptive update law \cite{Goodwin1984}:
\begin{align}
\label{eq:updatelaw}
\hat\theta_d(k)&=\hat\theta_d(k-1)+\frac{a(k)\Phi_d(k-d)\varepsilon_d(k)}{1+\Phi_d(k-d)^T\Phi_d(k-d)}\\
\label{eq:avoidzero}a(k)&=\begin{cases}
1&\begin{minipage}{6cm}
if $\nu$-th element of right-hand side of (\ref{eq:updatelaw}) evaluated using $a(k)=1$ is $\neq 0$
\end{minipage}\\
\gamma_d&\text{otherwise, where }0<\gamma_d<2,\gamma_d\neq 1
\end{cases}\\
\varepsilon_d(k)&=y(k)-\hat\theta_d(k-1)^T\Phi_d(k-d),\label{eq:updateepsilon}
\end{align}
with $\hat\theta_d(k)=\left[\hat\vartheta_d(k)^T\;\;\hat\theta_{d,\nu}(k)\right]^T$. Equation (\ref{eq:avoidzero}) is necessary to avoid division by zero in the control law (\ref{eq:adaptivecontroller}). Theorem \ref{thm:fixeddelay} addresses the stability of the adaptive system given by (\ref{eq:predictorform}), (\ref{eq:adaptivecontroller}), and (\ref{eq:updatelaw})-(\ref{eq:updateepsilon}). The reader is referred to Theorem~6.3.1 in \cite{Goodwin1984} or Theorem~5.1 in \cite{Goodwin1980} for the proof of Theorem \ref{thm:fixeddelay}.

\begin{thm}
\label{thm:fixeddelay}
Let $D(k)\equiv 0$. Subject to Assumption \ref{ass:fixeddelay} and given a fixed delay $d$, the adaptive controller (\ref{eq:adaptivecontroller}) with the update law (\ref{eq:updatelaw}) guarantees that the plant given by (\ref{eq:predictorform}) follows the reference \yref, i.e., $\lim_{k\to\infty}\left(y(k)-\yref(k)\right)=0$, and that the sequences $\{\hat\theta_d(k)\}$, $\{y(k)\}$ and $\{u(k)\}$ are bounded for all $k$.
\end{thm}

\subsection{Persistent excitation and sufficient richness}

The following definitions related to persistent excitation are needed to introduce our switching controller. We define the terms \emph{persistently exciting} and \emph{sufficiently rich} in the following way:

\begin{definition}[\cite{Bai1985}]\label{def:pe}
A sequence $x(t)\in\R^n$ is said to be \emph{persistently exciting (PE) (in $N$ steps)}, if there exists $N\in\mathbb{Z}^+,\alpha>0$ such that
\begin{equation}
\sum_{t=t_0+1}^{t_0+N}x(t)x(t)^T\geqslant\alpha I
\end{equation}
uniformly in $t_0$.
\end{definition}

\begin{definition}[\cite{Bai1985}]\label{def:sr}
A sequence $x(t)\in\R^n$ is said to be \emph{sufficiently rich (SR) of order $m$ (in $N$ steps)}, if there exists $N\in\mathbb{Z}^+,\alpha>0$ such that
\begin{equation}
\sum_{t=t_0+1}^{t_0+N}\xi_m(t)\xi_m(t)^T\geqslant\alpha I
\end{equation}
with $\xi_m(t)=\begin{bmatrix}x(t+1)&x(t+2)&\cdots & x(t+m)\end{bmatrix}^T$ uniformly for all $t_0$.
\end{definition}

The following Lemma is useful to prove Theorem \ref{thm:persistentexcitation}.

\begin{lemma}[\cite{Bai1985}]
Suppose that $y_1(t)$ and $y_2(t)$ are two bounded sequences taking values in $\R^n$ satisfying $\sum_{t=1}^{\infty}\|y_1(t)-y_2(t)\|<\infty$. Then $y_1(t)$ is SR of order $p$ if and only if $y_2(t)$ is SR of order $p$.
\end{lemma}

The reader is referred to \cite{Bai1985} for the proof of Theorem \ref{thm:persistentexcitation}.

\begin{lemma}\label{lem:reachability}
Consider the discrete time system 
\begin{equation}\label{eq:reachable}
X(t+1)=AX(t)+BU(t)
\end{equation}
with $X(t)\in\R^n$, $A\in\R^{n\times n}$, $B\in\R^{n\times l}$, and $U(t)\in\R^l$. Assume that (\ref{eq:reachable}) is completely reachable and that the input $U(t)$ is SR of order $1\leq p\leq n$. Then,
\begin{equation}\label{eq:rankstate}
\rank\left(\sum_{t=t_0+1}^{t_0+n+N}X(t)X(t)^T\right)=p
\end{equation}
for all $t_0\geqslant 0$.
\end{lemma}

\begin{proof}
We first rewrite (\ref{eq:reachable}) as
\begin{equation}
X(t+j)=A^jX(t)+\sum_{i=1}^jA^{j-i}BU(t+i-1)
\end{equation}
and define the characteristic polynomial of $A$ to be
\begin{equation}
p(z)=z^n+a_1+z^{n-1}+\ldots+a_n
\end{equation}
and let
\begin{equation}
V(t)=X(t+n)+a_1X(t+n-1)+\ldots+a_nX(t).
\end{equation}
Then, from the Cayley-Hamilton theorem it follows that
\begin{equation}
V(t)=G \begin{bmatrix}U^T(t)&\ldots & U^T(t+n-1)\end{bmatrix}
\end{equation}
where $G=\begin{bmatrix}A^{n-1}B+a_1A^{n-2}B+\ldots+a_{n-1}B,\ldots,B\end{bmatrix}$. Let
\begin{equation}
Y(t_0)=\begin{bmatrix}V(t_0+1)&V(t_0+2)&\ldots&V(t_0+N)\end{bmatrix}.
\end{equation}
Then,
\begin{equation}\label{eq:kleinergamma}
Y(t_0)Y^T(t_0)=\sum_{t=t_0+1}^{t_0+N}V(t)V^T(t)\geqslant\alpha\gamma(I_p\oplus 0)
\end{equation}
where "$\oplus$" denotes the direct sum, $I_p$ is the $p\times p$ identity matrix, and $0<\gamma=\sigma_{\min}(GG^T)$ is the minimal singular value of $GG^T$. $Y(t_0)$ can be also stated in terms of $X(t)$ in the following way
\begin{align}
Y(t_0)&=\begin{bmatrix}X(t_0+1),X(t_0+2),\ldots,X(t_0+n+N)\end{bmatrix}P\\
&=WP
\end{align}
where $P\in\R^{(N+n)\times N}$ is given by
\begin{equation}
P=\begin{bmatrix}
a_n&0&\cdots&0\\
\vdots&a_n&\ddots&\vdots\\
1&\vdots&\ddots&0\\
0&1&&a_n\\
\vdots&\ddots&\ddots&\vdots\\
0&\cdots&0&1
\end{bmatrix}
\end{equation}
Then,
\begin{equation}
Y(t_0)Y^T(t_0)=WPP^TW^T\leqslant\sigma_{\max}(PP^T)WW^T
\end{equation}
and hence using (\ref{eq:kleinergamma}),
\begin{equation}
\sum_{t=t_0+1}^{t_0+n+N}X(t)X(t)^T=WW^T\geqslant\frac{\alpha\gamma}{\sigma_{\max}(PP^T)}(I_p\oplus 0)
\end{equation}
That is,
\begin{equation}
\rank\left(\sum_{t=t_0+1}^{t_0+n+N}X(t)X(t)^T\right)\geqslant p
\end{equation}
From Theorem 1 in \cite{Bai1985}, it follows that an input signal which is SR of order $p$ implies a persistent excitation of at most $p$ directions in a $n$-dimensional space. Thus,
\begin{equation}
\rank\left(\sum_{t=t_0+1}^{t_0+n+N}X(t)X(t)^T\right)=p
\end{equation}
\end{proof}

\emph{Remark:} Much of the existing results pertaining to persistent excitation pertain to the case when the external input $U(t)$ is SR of order $n$. Lemma \ref{lem:reachability} above as well as Corollary \ref{cor:inomega} stated below address the case when $U(t)$ is SR of order $p$, where $p\in (1,n)$, which to our knowledge has not been examined in the literature. As our goal is tracking of an arbitrary signal and not identification, we do not need the SR-order to be $n$, but arbitrary and fixed at some $p$. 

\begin{cor}\label{cor:inomega}
Consider the discrete time system 
\begin{equation}\label{eq:reachable2}
X(t+1)=AX(t)+BU(t)
\end{equation}
with $X(t)\in\R^n$, $A\in\R^{n\times n}$, $B\in\R^{n\times l}$, and $U(t)\in\R^l$. Assume that (\ref{eq:reachable2}) is completely reachable and that the input $U(t)$ is SR of a fixed order $1\leq p\leq n$. Then, there exists a subspace $\Omega_p\subset\R^n$ such that
\begin{equation}\label{eq:dimmx}
X(t)\in\Omega_p\;\forall t\geq t_0\quad\text{with }\dimm\Omega_p=p.
\end{equation}
That is, the columns of $X(t)$ span the subspace $\Omega_p$.
\end{cor}

\begin{proof}
This follows directly from Lemma \ref{lem:reachability} and the fact that for any complex matrix $\Gamma$ the following is true
\begin{equation}
\rank(\Gamma)=\dimm\,\imm\,\Gamma
\end{equation}
where $\imm$ denotes the image of the linear transformation $\Gamma$.
\end{proof}

We make the following assumption which refers to an invariant property of persistent excitation.

\begin{ass}\label{ass:pe}
$\yref(k)$ is sufficiently rich of constant order $1\leq p\leq M$ for all $k$.
\end{ass}

Theorem \ref{thm:persistentexcitation} connects the sufficient richness of \yref with the tracking error and the parameter convergence in an adaptive system.

\begin{thm}\label{thm:persistentexcitation}
Let $D(k)\equiv 0$.  Suppose the adaptive controller (\ref{eq:adaptivecontroller})-(\ref{eq:updateepsilon}) is used to control the plant in (\ref{eq:thetaregressor}) and let Assumptions \ref{ass:fixeddelay} and \ref{ass:pe} hold.  Then
\begin{enumerate}
\item[(i)] $\lim_{k\to\infty}e(k)=\lim_{k\to\infty}y(k)-\yref(k)=0$, and
\item[(ii)] $\Phi_d(k)\in\Omega_p\subset\R^M$ as $k\to\infty$
\item[(iii)] $\tilde\theta_d(k)=\hat\theta_d(k)-\theta_d^*$ converges to $\bar\Omega_{M-p}$ where $\bar\Omega_{M-p}$ is defined as
\begin{equation}
\bar\Omega_{M-p}\defi\left\{x\;|\;\Phi_d(k)^Tx=0\text{ for }k\to\infty\right\}
\end{equation}
where $\Phi_d(k)$ is given in (\ref{eq:phiandtheta}).
\end{enumerate}
\end{thm}

\begin{proof}
Item (i) follows directly from Theorem \ref{thm:fixeddelay} as it is independent of any persistent excitation of the reference signal \yref. Item (ii) follows by noting that the adaptive system in (\ref{eq:model}) and (\ref{eq:adaptivecontroller})-(\ref{eq:updateepsilon}) becomes asymptotically linear, and this linear system in turn has a state that satisfies (\ref{eq:dimmx}) due to Assumption \ref{ass:pe}. Item (iii) follows from (i) and the fact that $e(k)=\Phi_d(k-d)^T\tilde\theta_d(k)$.
\end{proof}

\section{The switching adaptive controller}
\label{sec:switchingadaptivecontroller}

\subsection{Hybrid Communication Protocols}
\label{sec:hybridcommunicationprotocols}
Hybrid communication protocols such as FlexRay \cite{flexray} provide time-triggered and event-triggered bus schedules. Time-triggered communication offers highly predictable temporal behavior, and event-triggered communication provides efficient bandwidth usage. To exploit their combined advantages, we propose the use of a hybrid communication protocol in this paper.

To illustrate our proposed scheme, we use FlexRay as it has been established as the de-facto standard for future automotive in-vehicle networks. The FlexRay protocol is organized in a sequence of communication cycles of fixed length. Further, every such cycle is subdivided into a \emph{static} segment (ST) and a \emph{dynamic} segment (DYN). The static segment is partitioned into time windows of fixed and equal length which are referred as \emph{slots}. Each processing unit is assigned one or more slots indexed by a slot number $S\in\mathcal{S}_{\text{ST}}$ that indicates available time windows for bus access in the static segment. Due to the predictable temporal behavior we use the static segment schedules for communication in the time-triggered mode and dynamic segment schedules in the event-triggered mode.

The dynamic segment is partitioned into \emph{minislots} of much smaller duration than the static slots. Similar to the static segment, the minislots are indexed by a slot number to indicate allowable message transmissions. However, dynamic slots are of varying size depending on the size of the message which is transmitted in a certain slot $S\in\mathcal{S}_{\text{DYN}}$, where $\mathcal{S}_{\text{DYN}}$ is the set of available slot numbers in the dynamic segment. If no message is ready for transmission in a particular slot only one small minislot is consumed and the slot number is incremented with the next minislot. However, if a message is transmitted in a slot $S\in\mathcal{S}_{\text{DYN}}$ then the slot number increments with the next minislot after which the message transmission has been completed. Hence, bus resources are only utilized if messages are actually transmitted on the bus; otherwise only one minislot is consumed. Dynamic segment schedules are used for communication in the event-triggered~mode.

The focus of this problem is the simultaneous control of several applications for stabilization. That is, the goal is to choose $u$, the input of the $i$th control application such that $y(k)$, its output, converges to $\yref(k)$ which is zero. In the context of the problem under consideration, all control applications are partitioned into a sensor task $T_s$, a controller task $T_c$, and an actuator task $T_a$ (Figure \ref{fig:bus}). We consider a communication protocol where each communication cycle is divided into time-triggered and event-triggered segments. Using \emph{time-triggered} communication schedules, denoted as \Mtt, applications are allowed to send messages only at their assigned slots and the tasks are triggered synchronously with the bus, i.e., we assume that the communication delay due to the finite speed of the bus is negligible and hence the delay $d$ in (\ref{eq:predictorform}) is equal to $1$. On the other hand, in an \emph{event-triggered} schedule, denoted as \Met, the tasks are assigned priorities in order to arbitrate for access to the bus. Note that in our setup, multiple control applications share the same bus and hence multiple control messages have to be sent using a common bus and thus the messages might experience a communication delay $\tau$ when the higher priority tasks access the event-triggered segment. We choose the event-triggered communication schedules such that the sensor-to-actuator delay $\tau$ is within $(d_2-1)$ sample intervals, i.e., $0<\tau\leq (d_2-1)h$ for the control-related messages and hence the delay $d$ is at most equal to $d_2$ with $d_2\geqslant 2$. In summary, the delay $d=1$ if $\Mbus=\Mtt$ and $d=d_2$ if $\Mbus=\Met$ where \Mbus denotes the protocol used at time $k$.

The properties of the varying delay of the TT and ET protocol are directly exploited in the control design in the following way. Whenever the error between the plant output and its desired value is above some threshold \eth, we send the control messages over the TT protocol, as this guarantees an aggressive control action with minimal communication delay. Otherwise, the control messages are sent using the ET protocol. That is,
\begin{equation}
\Mbus=\begin{cases}
\Met&\text{if }|y(k)-\yref|\leqslant\eth\\
\Mtt&\text{if }|y(k)-\yref|>\eth.
\end{cases}
\label{eq:protocol}
\end{equation}
That is, the protocol switches depending on the state of the control application, as in (\ref{eq:protocol}).

\begin{figure}[tbp]
  \centering
  \psfrag{P1}[cc][cc]{$PU_1$}
  \psfrag{P2}[cc][cc]{$PU_2$}
  \psfrag{Ts}[cc][cc]{$T_s$}
  \psfrag{Tc}[cc][cc]{$T_c$}
  \psfrag{Ta}[cc][cc]{$T_a$}
  \psfrag{y}[cc][][0.8]{$y(k)$}
  \psfrag{sensor}[tc][cc][0.7]{sensor}
  \psfrag{controller}[cc][cc][0.7]{controller}
  \psfrag{actuator}[cc][cc][0.7]{actuator}
  \psfrag{shared}[cc][cc]{shared communication network}
  \includegraphics[width=0.8\linewidth]{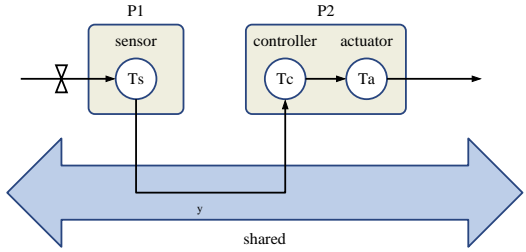}
  \caption{Schematic of a cyber-physical control system}
  \label{fig:bus}
\end{figure}

\subsection{Controller design}
\label{sec:switchingcontrollerdesign}
Commensurate with the switching protocol in (\ref{eq:protocol}), we propose a switch in the adaptive controller as well, and is defined below:
\begin{gather}
  \label{eq:TTcontroller}
  \left.\begin{array}{l}
    \displaystyle u(k)=\frac{1}{\hat\theta_{1,\nu}(k)}\left(\yref(k+1)-\hat\vartheta_1(k)^T\phi_1(k)\right)\\
    \varepsilon_1(k)=y(k)-\hat\theta_1(k-1)^T\Phi_1(k-1)\\
    \displaystyle\hat\theta_1(k)=\hat\theta_1(k-1)+\frac{a(k)\Phi_1(k-1)\varepsilon_1(k)}{1+\Phi_1(k-1)^T\Phi_1(k-1)}\\
        a(k)=\begin{cases}
1&\begin{minipage}{5cm}
if $\nu$-th element of right-hand side of update law evaluated using $a(k)=1$ is $\neq 0$
\end{minipage}\\
\gamma_1&\text{otherwise, where }0<\gamma_1<2,\gamma_1\neq 1
\end{cases}
  \end{array}\right\}\hspace*{-0.3cm}\begin{array}{l}\text{if}\\\Mtt \end{array}
\end{gather}
where $\phi_1(k)$ is given in Eq. (\ref{eq:phiandthetaklein}), $\Phi_1(k)$ is given in Eq. (\ref{eq:phiandtheta}), $\hat\theta_1(k)=\left[\hat\vartheta_1(k)^T\;\;\hat\theta_{1,\nu}(k)\right]^T$ is the estimation of the controller gains $\theta_1^*$ (Eq. \ref{eq:phiandtheta}), and $\gamma_1\in (0,2)$.

If $\mathcal{M}_{\text{Bus}}(k)=\Met$, the adaptive controller is given by
\begin{equation}
  \label{eq:ETcontroller}
  \left.\begin{array}{l}
    \displaystyle u(k)=\frac{1}{\hat\theta_{2,\nu}(k)}\left(\yref(k+d_2)-\hat\vartheta_2(k)^T\phi_2(k)\right)\\
    \varepsilon_2(k)=y(k)-\hat\theta_2(k-1)^T\Phi_2(k-2)\\
    \displaystyle\hat\theta_2(k)=\hat\theta_2(k-1)+\frac{a(k)\Phi_2(k-d_2)\varepsilon_2(k)}{1+\Phi_2(k-d_2)^T\Phi_2(k-d_2)}\\
    a(k)=\begin{cases}
1&\begin{minipage}{5cm}
if $\nu$-th element of right-hand side of update law evaluated using $a(k)=1$ is $\neq 0$
\end{minipage}\\
\gamma_2&\text{otherwise, where }0<\gamma_2<2,\gamma_2\neq 1
\end{cases}
  \end{array}\right\}\hspace*{-0.3cm}\begin{array}{l}\text{ if }\\\Met\end{array}
\end{equation}
where $\phi_2(k)$ is given in Eq. (\ref{eq:phiandthetaklein}), $\Phi_2(k)$ is given in Eq. (\ref{eq:phiandtheta}), $\hat\theta_2(k)=\left[\hat\vartheta_2(k)^T\;\;\hat\theta_{2,\nu}(k)\right]^T$ is the estimation of the controller gains $\theta_2^*$ (Eq. \ref{eq:phiandtheta}), and $\gamma_2\in (0,2)$.

\subsection{Main Result}
The following definitions are useful for the rest of the paper. We denote the instants of time when the switch from TT to ET occurs with $k_p$, $p=1,3,5,\ldots$, and the instants of time when the switch from ET to TT occurs with $k_p$, $p=2,4,6,\ldots$. That is, the TT protocol is applied for $k\in[k_{2p}';k_{2p+1}],p\in\NN_0$ and the ET protocol is applied for $k\in[k_{2p+1}';k_{2p}],p\in\NN_0$ with $k_p'\defi k_p+1$ and switches occurring between $[k_p;k_p'],p\in\NN$ (see Figure \ref{fig:error}).

\begin{ass}\label{ass:Dimpulse}
The disturbance $D(k)$ in (\ref{eq:predictorform}) is an impulse train, with the distance between any two consecutive impulses greater than a constant \Tdw.
\end{ass}

This is the main result of the paper:
\begin{thm}
\label{thm:betaknown}
Let the plant and disturbance $D$ in (\ref{eq:predictorform}) satisfy Assumptions \ref{ass:fixeddelay}, \ref{ass:pe}, and \ref{ass:Dimpulse}. Consider the switching adaptive controller in (\ref{eq:TTcontroller}) and (\ref{eq:ETcontroller}) with the hybrid protocol in (\ref{eq:protocol}) and the following parameter estimate selections at the switching instants
\begin{align}
\label{eq:thetachoice1}\hat\theta_1(k_p)   & =0,                     && p=0,2,4,\ldots\\
\label{eq:thetachoice0}\hat\theta_2(k_1)   & =0\\
\label{eq:thetachoice2}\hat\theta_2(k_p+l) & =\hat\theta_2(k_{p-1}), && p=3,5,7,\ldots,\\
                                           &                         && l=0,1,\ldots, m_2+d-1
\end{align}
Then there exists a positive constant $\Tdws$ such that for all $\Tdw\geq\Tdws$, the closed loop system has globally bounded solutions.
\end{thm}

A qualitative proof of Theorem \ref{thm:betaknown} is as follows:\\ 
First, Theorem \ref{thm:fixeddelay} shows that if either of the individual control strategies (\ref{eq:TTcontroller}) or (\ref{eq:ETcontroller}) is deployed, then boundedness is guaranteed. That is, for a sufficiently large dwell time \Tdw over which the controller stays in the TT protocol, with the controller in (\ref{eq:TTcontroller}), boundedness can be shown. After a finite number of switches, when the system switches to an ET protocol, it is shown that the regressor vector remains in the same subspace as in the earlier switch to ET and hence, the corresponding tracking error remains small even after the switch to ET. Hence the stay in ET is ensured for a finite time, guaranteeing boundedness with the overall switching controller. 

\textit{Proof of Theorem \ref{thm:betaknown}:} We define an equivalent reference signal \yrefp that combines the effect of both \yref and the disturbance $D$ as
\begin{equation}
\yrefp(k)\defi\yref(k)+D'(k)
\label{eq:yrefp}
\end{equation}
where $D'(k)$ is given by
\begin{equation}
D'(k)\defi G^{-1}(\qq)D(k).
\label{eq:dprime}
\end{equation}
and $G(\qq)=\frac{B(\qq)}{A(\qq)}$ is the transfer function of the plant (\ref{eq:model}).
Also, we define a reference model signal $\phi_i^*$ given by 
\begin{equation}\label{eq:phistern}
\phi_i^*(\qq)=G_i(1+\vartheta_i^{*T}G_i)^{-1}\yrefp
\end{equation}
where the transfer functions $G_i,i=1,2$ is given by
\begin{equation}
G_i(\qq)=\begin{bmatrix}
W(\qq)\qq\\\vdots\\W(\qq)\qn{m_1+1}\\\qq\\\vdots\\\qn{m_2-d_i+1}
\end{bmatrix}
\label{eq:gijoe}
\end{equation}
and the optimal feedback gain $\vartheta_i^*,i=1,2$ is given by (\ref{eq:phiandthetaklein}). The overall ideal closed-loop system is given by the block diagram shown in Figure \ref{fig:blockdiagram01}. We note that when there is no disturbance, the output $\phi_i^*$ corresponds to the desired regressor vector, and its first element of the vector corresponds to \yref.

\begin{figure}[htbp]
\begin{center}
\includegraphics[width=0.8\linewidth]{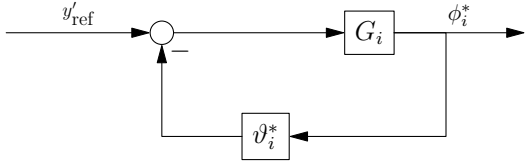}
\end{center}
\caption{Block diagram of the reference model for $i=1,2$}
\label{fig:blockdiagram01}
\end{figure}

When the algorithm is in mode \Mtt, the underlying error equation is given by
\begin{equation}
\label{eq:errorTT1}
    e_1(k+1)=\left(\vartheta_1^*-\hat\vartheta_1(k)\right)^T\phi_1(k)=\tilde\vartheta_1(k)^T\phi_1(k).
\end{equation}
with $\tilde\vartheta_1(k)=\vartheta_1^*-\hat\vartheta_1(k)$. When the system is in mode \Met, the error equation is given by
\begin{equation}
\label{eq:errorTT2}
    e_2(k+2)=\left(\vartheta_2^*-\hat\vartheta_2(k)\right)^T\phi_2(k)=\tilde\vartheta_2(k)^T\phi_2(k).
\end{equation}
with $\tilde\theta_2(k)=\theta_2^*-\bar\theta_2(k)$.

Define $\theta^*_a$ as
\begin{equation}
\theta_a^*=\begin{cases}\begin{bmatrix}\theta_1^{*T}&0\end{bmatrix}^T&\text{if }\Mbus=\Mtt\\\theta_2^*&\text{if }\Mbus=\Met\end{cases}
\end{equation}
Choose Lyapunov function $V(k)=\tilde\theta_a(k)^T\tilde\theta_a(k)$ where $\tilde\theta_a(k)=\theta_a^*-\hat\theta_a(k)$. Let $\Delta V(k)=V(k)-V(k-1)$.

The proof consists of the following four stages:

\newcounter{Cnt_stages}
\newcounter{Cnt_steps}
\begin{list}{\textbf{\textsc{Stage \arabic{Cnt_stages}:}}}{\usecounter{Cnt_stages}%
\setlength{\leftmargin}{0pt}%
\setlength{\itemindent}{0pt}%
\setlength{\labelsep}{10pt}%
\setlength{\labelwidth}{-10pt} } 
\item Let there exist a sequence of finite switching times $\{k_l\}_{l\in\NN}$ with the properties described above. Then the errors $e_1(k)$ and $e_2(k)$ are bounded for all $k$.\\The proof of Stage 1 is established using the following three steps:
	\begin{list}{\textbf{Step \arabic{Cnt_stages}-\arabic{Cnt_steps}}}{\usecounter{Cnt_steps}\setlength{\leftmargin}{20pt}}%
	\item There exists a $\Delta\in\mathbb{N}$ such that $\forall\varepsilon\in\;]0;\eth]:\;|e_1(k_1)|<\varepsilon\leqslant\eth$ where $k_1=k_0+\Delta+m_1$.
	\item During \Mtt (\Met), the error $e_1(k)$ ($e_2(k)$) is bounded.
	\item There exists a constant $M_1<\infty$ with $|e_1(k_p')|\leqslant M_1$, for $p=2,4,6,\ldots$.
	\end{list}
\item The length of the interval $[k_p';k_{p+1}]$ is greater than 2, i.e., $k_{p+1}-k_p'>2,p=1,3,5,\ldots$\\
Stage 2 is established using the following steps:
	\begin{list}{\textbf{Step \arabic{Cnt_stages}-\arabic{Cnt_steps}}}{\usecounter{Cnt_steps}\setlength{\leftmargin}{20pt}}
	\item If $\lim_{k\to\infty}|y(k)-y_{\text{{\rm ref}}}(k)|=0$ then $\lim_{k\to\infty}\|\phi_i(k)-\phi^*_i(k)\|=0$
	\item If $|y(k)-\yref(k)|\to 0$, then $\phi_i(k_j)\in\Omega_i$
	\item $|e_2(k_p')|=|\tilde\vartheta_2(k_{p-1})^T\phi_2(k_p'-d_2)+\tilde\theta_{2,\nu}(k_{p-1})\yref(k_p')|\leqslant\eth$ for $p=3,5,7,\ldots$
	\item $|e_2(k_p'+1)|\leqslant\eth$ for $p=3,5,7,\ldots$
	\end{list}
\item $V(k)$ is bounded for all $k\in\NN_0$.\\
The following steps will be used to establish Stage 3:
	\begin{list}{\textbf{Step \arabic{Cnt_stages}-\arabic{Cnt_steps}}}{\usecounter{Cnt_steps}\setlength{\leftmargin}{20pt}}
	\item $\Delta V(k_p')\leq M_2<\infty$ and $\Delta V(k_{p+1}')\leq M_3<\infty$, for all $p\in\NN_0$.
	\item $\Delta V(k)\leqslant 0$ during \Mtt\ and during \Met
	\end{list}
\item The control input is bounded for all $k$ and hence all signals are bounded.\\
The following two steps will be used to prove Stage 4:
	\begin{list}{\textbf{Step \arabic{Cnt_stages}-\arabic{Cnt_steps}}}{\usecounter{Cnt_steps}\setlength{\leftmargin}{20pt}}
	\item $|u(k)|\leqslant M_4\quad\forall k$
	\item All signals are bounded.
	\end{list}
\end{list}

We note that the proofs of Stages 1, 3, and 4 are identical to that in \cite{Voit2012} and are therefore omitted here. Since Stage~2 differs significantly from its counterpart in \cite{Voit2012} due to $\yref\neq 0$, we provide its proof in detail below.

\begin{figure}[tbp]
\psfrag{eth}{\eth}
\psfrag{e}{$\varepsilon$}
\psfrag{e1}{$e_1$}
\psfrag{e2}{$e_2$}
\psfrag{k0}{$k_0$}
\psfrag{k1}{$k_0'$}
\psfrag{k2}{$k_1$}
\psfrag{k3}{$k_1'$}
\psfrag{k4}{$k_2$}
\psfrag{k5}{$k_2'$}
\psfrag{k6}{$k_3$}
\psfrag{k7}{$k_3'$}
\psfrag{ET}{ET}
\psfrag{TT}{TT}
\psfrag{Sw}{switch (S)}
\psfrag{S}{S}
\psfrag{error}{error}
\psfrag{time}[][][0.8]{time}
\includegraphics[width=1.0\linewidth]{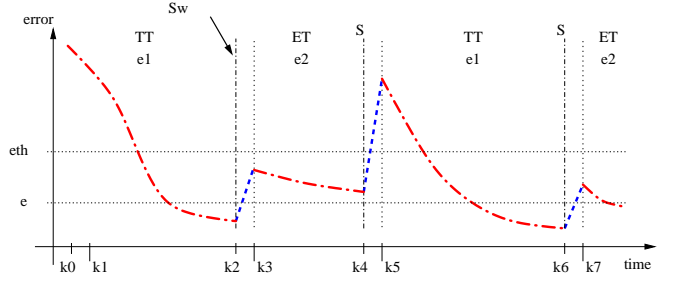}
\caption{Schematic evolution of the error with a given sequence of switching times. Impulses in $D(k)$ are assumed to occur at $k_{p}, p=0,2,4,...$.}
\label{fig:error}
\end{figure}

\stepcounter{stepstages}\setcounter{stepproof}{0}

\begin{proofstep}{$\lim_{k\to\infty}|y(k)-\yref(k)|=0\Rightarrow\lim_{k\to\infty}\|\phi_i(k)-\phi^*_i(k)\|=0$}
In this step we show that if the tracking error $|y(k)-\yref(k)|$ is small the state signal error $\|\phi_i(k)-\phi^*_i(k)\|$ is also small.

The signal $\phi_i(k)-\phi^*_i(k)$ is the output produced by the following transfer function $H$ with $|y(k)-\yref(k)|$ as the input:
\begin{equation}
H=\begin{bmatrix}
1\\\qq\\\vdots\\\qn{m_1+1}\\W^{-1}(\qq)\qq\\W^{-1}(\qq)\qn{2}\\\vdots\\W^{-1}(\qq)\qn{m_2-d_i+1}
\end{bmatrix}
\label{eq:H}
\end{equation}
where $W^{-1}(\qq)$ is the inverse of the plant transfer function $W(\qq)=\frac{B(\qq)}{A(\qq)}$ with the input signal $y(k)-\yref(k)$ and $\phi_i^*(k)$ given in (\ref{eq:phistern}). From Assumption \ref{ass:fixeddelay}, it follows that $W^{-1}(\qq)$ is a stable transfer function. Hence, as $|y(k)-\yref(k)|$ tends to zero, $\phi_i(k)-\phi^*_i(k)$ also tends to zero.\end{proofstep}

\begin{proofstep}{If $|y(k)-\yref(k)|\to 0$, then $\phi_i(k_j)\in\Omega_i$}
We first show that $\phi_d^*(k)\in\Omega_i$ for $i=1,2$ and $j=1,2,3,\ldots$. We note that the reference model given in (\ref{eq:phistern}) is a linear system and hence there exists a state space representation
\begin{equation}
\phi_i^*(k+1)=R\phi_i^*(k)+S\yrefp
\end{equation}
with $(R,S)$ being completely reachable. Then it follows directly from Lemma \ref{cor:inomega} that $phi_i^*(k_j)\in\Omega_i$ for $i=1,2$ and $j=1,2,3,\ldots$. Together with Step 2-1 it follows that if $|y(k)-\yref(k)|\to 0$, then $\phi_i(k_j)\in\Omega_i$.
\end{proofstep}

\begin{proofstep}{$|e_2(k_p')|=|\tilde\vartheta_2(k_{p-1})^T\phi_2(k_p'-d_2)+\tilde\theta_{2,\nu}(k_{p-1})\yref(k_p')|\leqslant\eth$ for $p=3,5,7,\ldots$}
First, we show that the error of the signal generated by the reference model signal $\phi^*_2(k_p'-d_2)$ together with the last parameter estimation value $\vartheta_2(k_{p-1})$ at the end of the previous ET phase is small and therefore the output error $e_2(k_p'),p=3,5,7,\ldots$ is below the threshold \eth.

From Step 2-2 we know that $\phi^*_2(k_p'-d_2)$ is in the same subspace $\Omega_2$ as $\phi^*_2(k_{p-1})$. From Step 2-1 we know that $\phi^*_2(k_{p-1})$ is close to $\phi_2(k_{p-1})$ which in turn generates together with $\tilde\vartheta_2(k_{p-1})$ and $\tilde\theta_{2,\nu}(k_{p-1})$ an error which is $\leqslant\varepsilon$ according to Theorem \ref{thm:fixeddelay}. Hence,
\[|\tilde\vartheta_2(k_{p-1})^T\phi^*_2(k_p'-d_2)|+\tilde\theta_{2,\nu}(k_{p-1})\yref(k_p')\leqslant\varepsilon.\]

From Step 2-1 we know that $\phi_2(k_p'-d_2)$ is close to $\phi^*_2(k_p'-d_2)$. Hence, according to Step 2-4 we have $|e_2(k_p')|\leqslant\eth,p=3,5,7,\ldots$.
\end{proofstep}

\begin{proofstep}{$|e_2(k_p'+l)|\leqslant\eth, p=3,5,7,\ldots$ and $l=1,2,\ldots,m_2+d_2-1$}
This step shows that the error at the beginning of the ET mode is below the threshold for at least $m_2+d_2$ steps.

From Step 2-3 we know that $|e_2(k_p')|\leqslant\eth$. According to the parameter choice in (\ref{eq:thetachoice2}), the controller uses a constant initial value for the first $m_2+d_2-1$ steps. Thus, the error $|e_2(k_p'+l)|=|\tilde\vartheta_2(k_{p-1})^T\phi_2(k_p'+l-d_2)+\tilde\theta_{2,\nu}(k_{p-1})\yref(k_p'+l)|\leqslant\eth$ because Steps 2-1 to 2-5 can be applied.
\end{proofstep}

\subsection{Comments on the Main Result}
\label{sec:comments}
Theorem \ref{thm:betaknown} implies that the plant in (\ref{eq:predictorform}) can be guaranteed to have bounded solutions with the proposed adaptive switching controller in (\ref{eq:TTcontroller}) and (\ref{eq:ETcontroller}) and the hybrid protocol in (\ref{eq:protocol}), in the presence of disturbances. The latter is assumed to consist of impulse-trains, with their inter-arrival lower bounded. We note that if no disturbances occur, then the choice of the algorithm in (\ref{eq:protocol}) implies that these switches cease to exist, and the event-triggered protocol continues to be applied. And switching continues to occur with the onset of disturbances, with Theorem \ref{thm:betaknown} guaranteeing that all signals remain bounded with the tracking errors $e$ converging to $\eth$ before the next disturbance occurs.

The nature of the proof is similar to that of all switching systems, in some respects. A common Lyapunov function $V(k)$ was used to show the boundedness of parameter estimates, which are a part of the state of the overall system (in Stage 3). The additional states were shown to be bounded using the boundedness of the tracking errors $e_1$ and $e_2$ (in Stage 1) and the control input using the method of induction (in Stage 4). Since the switching instants themselves were functions of the states of the closed-loop system, we needed to show that indeed these switching sequences exist, which was demonstrated in Stage 2. To this end, the sufficient richness properties of the reference signal were utilized to show that the signal vectors of a reference model and the system converge to the same subspace. Next, it was shown that the error generated by the reference model is small and thus concluded that the tracking error at the switch from TT to ET stays below the threshold \eth. It is the latter that distinguishes the adaptive controller proposed in this paper, as well as the methodology used for the proof, from existing adaptive switching controllers and their proofs in the literature.

\section{Summary}
\label{sec:conclusion}
In this work we considered the control of multiple control applications using a hybrid communication protocol for sending control-related messages. These protocols switch between time-triggered and event-triggered methods, with the switches dependent on the closed-loop performance, leading to a co-design of the controller and the communication architecture. In particular, this co-design consisted of switching between a TT and ET protocol depending on the amplitude of the tracking error, and correspondingly between two different adaptive controllers that are predicated on the resident delay associated with each of these protocols. These delays were assumed to be fixed and equal to $1$ for the TT protocol and greater than $2$ for the ET protocol. It was shown that for any reference input whose order of sufficient richness stays constant, the signal vector and the parameter error vector converge to subspaces which are orthogonal to each other. The overall adaptive switching system was shown to track such reference signals, with all solutions remaining globally bounded, in the presence of an impulse-train of disturbances with the inter-arrival time between any two impulses greater than a finite constant.

\footnotesize{
\bibliographystyle{IEEEtran}
\bibliography{codesign}
}

\end{document}